\documentclass[12pt]{article}
\usepackage{amsmath}
\usepackage{amsfonts}
\usepackage{amssymb}
\usepackage{latexsym}
\usepackage{amscd}
\usepackage{color}
\usepackage[T1]{fontenc}

\input amssym.def
\newcommand{\norma}[1]{\| #1 \|}
\newcommand{\conj}[2]{\left \{ {#1} \, : \, {#2} \right \}}

\newcommand{\cv}{\rightarrow}
\newcommand{\cvf}{\overset{\omega}{\rightarrow}}

\newcommand {\cvfe} {\overset{\omega^\ast}{\rightarrow}}

\newcommand {\N} {\mathbb{N}}

\newcommand{\sol}[1]{\text{sol}(#1)}

\newcommand{\proof}{\noindent \textit{Proof. }}
\newcommand{\qed}{\hfill \ensuremath{\Box}}

\newtheorem{df}{Definition}[section]

\newtheorem{prop}[df]{Proposition}
\newtheorem{teo}[df]{Theorem}
\newtheorem{corol}[df]{Corollary}
\newtheorem{lemma}[df]{Lemma}
\newtheorem{ex}[df]{Example}

\newtheorem{rema}[df]{Remark}
\newtheorem{remas}[df]{Remarks}

\begin{document}

\title{The property (d) and the almost limited completely continuous operators}
\author{M. L. Louren\c co and V. C. C. Miranda \footnote{Supported by a CAPES PhD's scholarship (88882.377946/2019-01).}}
\date{}

\maketitle

\begin{abstract}
In this paper, we study some geometric  properties in Banach lattices and the class of almost limited completely continuous operators. For example,  we study Banach lattices with property (d) and we give a new characterization of this property in terms of the solid hull of almost limited sets. 
     

\end{abstract}

\section{Introduction}

Throughout this paper $X$ and $Y$ will denote real Banach spaces, $E$ and $F$ will denote Banach lattices. We denote by $B_X$ the closed unit ball of $X$. In a Banach lattice the additional lattice structure provides a large number of tools that are not available in more general Banach spaces.
This fact facilitates the study of geometric properties of Banach lattices, as well as of
the properties of operators acting between lattices. A  number of questions arise naturally about the
relationship between the properties of the operators and the lattice structure. Among these questions,
one of the most frequent is the so-called majorization problem, which consists of, given positive operators $ S \leq  T$ between Banach lattices $E$ and $F$,  to study which properties $ S$ inherits from the operator $T.$
 These theories have won notoriety in many articles that have been developed over the last 20 years. Our interest is to study some geometric properties of Banach spaces in the context of Banach lattices and their consequences in operator theory.  To state our results, we need to fix some notation and recall some definitions.
We denote by $E^+$ the set of all positive elements in $E$. For a set $A \subset E$, $A^+ = A \cap E^+$.
Recall that $A \subset E$ is said to be \textit{solid} if $|y| \leq |x|$ with $x \in A$ implies that $y \in A$. In addition, if $F$ is a solid subspace of $E$, we say that $F$ is an \textit{ideal}.
The \textit{solid hull of $A$}, $\sol{A}$, (resp., the \textit{ideal generated by $A$}, $E_A$ ,) is the smallest solid set that contains $A$ (resp., the smallest ideal that contains $A$). A \textit{band} $B$ in $E$ is an ideal such that $\sup(A) \in B$ for every subset $A \subset B$ which has supremum in $E$.
A norm $\Vert \,\,\, \Vert $ of a Banach lattice $E$ is called  \textit{order continuous} if    for each net $(x_{\alpha})$ such that $x_\alpha \downarrow 0$ in $E$ imply that $\Vert x_\alpha\Vert  \to 0,$  where the notation  $x_\alpha \downarrow 0$   means that the net $(x_\alpha)$ is decreasing  and  $\inf \{x_\alpha : \alpha \} = 0.$ 
There are some  equivalent statements of this property which can be found in \cite[Section 2.4]{meyer}.

A Banach lattice $E$ has  \textit{property (d)} if $|f_n| \cvfe 0$ for every disjoint weak* null sequence $(f_n) \subset E'$. This property was studied by Wnuk in \cite{wnuk}, however, this terminology was introduced later in \cite{elbour}.
In \cite{chen}, Chen, Chen and Ji show that  if $E$ is a $\sigma$-Dedekind complete Banach lattice, then $E$ has property (d). In \cite{wnuk} Wnuk showed that $\ell_\infty/c_0$ has property (d) but it is not $\sigma$-Dedekind complete. It is also clear that if the lattice operations in $E'$ are weak* sequentially continuous, $E$ has property (d). The reciprocal does not hold since $\ell_\infty$ has property (d), but the lattice operations of $\ell_\infty'$ are not weak* sequentially continuous. Indeed, by Josefson-Nussenzweig theorem there exists weak* null sequence $(f_n) \subset \ell_\infty'$ with $\norma{f_n} = 1$ for all $n$, however $(|f_n|)$ is not weak* null, because for $e = (1,1,1, \dots)$, we have that $|f_n|(e) = \sup_{x \in [-e,e]} |f_n(x)| = \norma{f_n} = 1$ for all $n$.

 We recall a few geometric definitions in Banach spaces and its lattice versions.
A subset $A \subset X$ is said to be \textit{limited} if every weak* null sequence in $X'$ converges uniformly to zero on $A$. We say that $X$ has the \textit{Dunford-Pettis* property} (DP*) if every relatively weakly compact set is limited.
If every limited set of $X$ is relatively compact it is said that $X$ has the \textit{Gelfand-Phillips property} (GP), or equivalently every limited weakly null sequence is norm null. Separable, reflexive and Schur Banach spaces have the GP property. The study and consequences of limited sets, DP* property and GP property in Banach spaces can be found, for example in  \cite{bourgdiest, cargalou, drew}. We observe that
a $\sigma$-Dedekind complete Banach lattice has the GP property if and only if it has order-continuous norm (see Theorem 4.5 in \cite{wnukbook}).

 Following the idea and technique of Bouras \cite{bouras} who considered the disjoint versions of the geometric properties of Banach spaces for lattices, J. X. Chen, Z. L Chen and G. X. Ji introduced the concept of almost limited sets in a Banach lattice \cite{chen}. A bounded set $A \subset E$ is an \textit{almost limited set} if every disjoint weak* null sequence in $E'$ converges uniformly to zero on $A$.
 We say that $E$ has the   \textit{weak Dunford-Pettis* property} (wDP*) if every relatively weakly compact set is almost limited and $E$ has    the \textit{strong Gelfand-Phillips property} (sGP) if every almost limited set is relatively compact.  It is known that  $c_0$ and $c$ are Banach lattices with the sGP (see pages 17 in \cite{ardakani}) without the wDP* property (see Remark 3.4 in \cite{chen} and Example 2.10 in \cite{ardakani}, resp.) and the Banach lattices $\ell_\infty$ and $L_1[0,1]$ have the wDP* (see Remark 3.4 in \cite{chen}) and do not have the sGP property (see page 19 in \cite{ardakani} for the fact that $L_1[0,1]$ does not enjoy the sGP property).

The aim of this paper is to study property (d). In particular, we improve a result given in \cite{micha}, about the solid  hull of an almost  limited set, and we also improve Theorem 3.10 in \cite{ardakani} concerning operators that maps almost limited sets onto
almost limited sets. As a consequence we give a new characterization of this property in terms of the solid hull of almost limited sets. We  also apply these results to establish some properties of the almost limited completely continuous operators  between Banach lattices,  including a majority problem concerning this class of operators.

We refer the reader  to \cite{alip, meyer} for Banach lattices theory and positive operators.

\section{Banach lattices with property (d)}

The property (d) and the almost limited sets are deeply connected. 
In \cite{micha}, the authors proved that for $\sigma$-Dedekind complete Banach lattices, the solid hull of an almost limited set is also  an almost limited set. In the next Lemma, we improve such result
by showing that it holds for Banach lattices with property (d).

\begin{lemma} \label{lemadom1}
Let $E$ be a Banach lattice with property (d). If  $A \subset E$ is almost limited, then $\sol{A}$ also is almost limited.
\end{lemma}

\proof
Assume that $\sol{A}$ is not almost limited. We are going to prove that $A$ cannot be almost limited. Since $\sol{A}$ is not almost limited,
there exist a disjoint weak* null sequence $(f_n) \subset E'$ and a sequence $(x_n) \subset \sol{A}$ such that $|f_n(x_n)| \geq \epsilon$ for all $n$. In particular, for each $n$, there exists $y_n \in A$ such that $|x_n| \leq |y_n|$. On the  other hand, it follows from Theorem 1.23 in \cite{alip} that for each $n$, there exists $g_n \in E'$ satisfying $|g_n| \leq |f_n|$ and $|f_n|(|y_n|) = g_n(y_n)$. Now, as  $(f_n)$ is a disjoint weak* null sequence in $E'$ and $E$ has property (d), $|f_n| \cvfe 0$. This implies that $g_n \cvfe 0$. 
However,
$ \epsilon \leq |f_n(x_n)| \leq |f_n|(|x_n|) \leq |f_n|(|y_n|) = g_n(y_n).$ So,
 $A$ cannot be almost limited, against the assumption. \qed
 
 \medskip

Lemma \ref{lemadom1} is not true for general  Banach lattices. In fact, the singleton $\{e\}$ in $c$, where $c$ denotes the Banach lattice of all real convergent sequences, is almost limited. However, its solid hull $\sol { \{e\} } = [-e,e] = B_c$ is not almost limited.  It is known that every order interval in a Banach lattice $E$ is almost limited if and only if $E$ has property (d) (see Proposition 2.3 in \cite{elbour2}). Using this fact and Lemma \ref{lemadom1}, we have the following result.

\begin{prop} \label{prop22}
For a Banach lattice $E$, the following assertions are equivalent:
\begin{enumerate}
    \item $E$ has property (d).
    \item Every order interval is almost limited.
    \item The solid hull of every almost limited set also is almost limited.
\end{enumerate}
\end{prop}


In \cite{ardakani}, the authors gave  an example where bounded linear operators do not map almost limited sets onto almost limited sets.
The next result improves Theorem 3.10 of \cite{ardakani}, where   $F$ is $\sigma$-Dedekind complete.

\begin{teo} \label{lemadom2}
Let $E$ and $F$ be Banach lattices. Let $T: E \to F$ be an order bounded operator   and let $A \subset E$ be an almost limited solid set. If $F$ has property (d), then $T(A)$ is an almost limited subset of $F$.
\end{teo}

\proof
By Theorem 2.7 of \cite{elbour2}, it suffices to prove that $f_n(Tx_n) \to 0$ for each disjoint sequence $(x_n) \subset A^+$  and each disjoint sequence $f_n \cvfe 0$ in $(F')^+$. Considering the sequences $(x_n) \subset E$ and $(T'f_n) \subset E'$, there exists a disjoint sequence $(g_n) \subset E'$ satisfying $|g_n| \leq |T'f_n|$ and $g_n(x_n) = T'f_n(x_n)$ for all $n$ (see page 77 in \cite{alip}). We claim  that $g_n \cvfe 0$. Indeed, if $x \in E^+$, since $T$ is order bounded, there exists $y \in F^+$ with $T[-x,x] \leq [-y,y]$. Hence
\begin{align*}
    |T'f_n|(x) & = \sup \conj{|T'f_n (u)|}{|u| \leq x} \leq \sup \conj{|f_n(v)|}{|v| \leq y}  = |f_n|(y) = f_n(y) \to 0.
\end{align*}
Now, since $x \in E^+$, it follows that 
$ |g_n(x)| \leq |g_n|(|x|) = |T'f_n|(|x|) \to 0,  $ which yields the claim.
As $A \subset E$ is almost limited, $(x_n) \subset A$ and $(g_n) \subset E'$ is a disjoint weak* null sequence, $g_n(x_n) \to 0$. Hence $T'f_n(x_n) \to 0.$
\qed

\medskip

Now we give   more examples of Banach lattices satisfying property (d). 
Let $(E_n)$ be a family of Riesz spaces, the direct  sum $\bigoplus_{n \in \N} E_n$ is a Riesz space with the following partial order relation 
$ (x_n) \leq (y_n) \iff x_n \leq y_n$ for all $n \in \N. $
Further, $(x_n) \vee (y_n) = (x_n \vee y_n)$ and $(x_n) \wedge (y_n) = (x_n \wedge y_n).$
If, in addition, each $E_n$ is a Banach lattice, we can consider the following Banach lattices:
$$ (\bigoplus_{n \in \N} E_n)_0 = \conj{(x_n)_n}{x_n \in E_n, \, \lim_n \norma{x_n} = 0} $$
and
$$ (\bigoplus_{n \in \N} E_n)_p = \conj{(x_n)_n}{x_n \in E_n, \, \sum_n \norma{x_n}^p < \infty}, $$
where $1 \leq p < \infty.$ In $(\bigoplus_{n \in \N} E_n)_0$
we consider the supremum norm, while in $(\bigoplus_{n \in \N} E_n)_p$ the p-norm. Moreover, $(\bigoplus_{n \in \N} E_n)_0' = (\bigoplus_{n \in \N} E_n)_1$ and $(\bigoplus_{n \in \N} E_n)_p' = (\bigoplus_{n \in \N} E_n)_q$ where $1/p + 1/q = 1$ (see Theorem 4.6 in \cite{alip}).
It is natural to ask under which conditions the above Banach lattices have property (d).

\begin{prop} \label{expropd2}
Let $(E_n)$ be a family of Banach lattices. Then  $E = (\bigoplus_{n \in \N} E_n)_0$ has property (d) if, and only if   $E_n$  has property (d) for all $n.$
\end{prop}

\proof Suppose that all $E_n$ have property (d). Let $(f_k) \subset E' = (\bigoplus_{n \in \N} E_n')_1$ be a disjoint weak* null sequence. In particular, if $f_k = (f_k ^n)_n$ with $f_k ^n \in E_n'$ for all $k \in \N$ and $n \in \N$, we have that $(f_k ^n)_k$ is disjoint in $E_n'$, because if $j \neq k$,
$ 0 = f_k \wedge f_j = (f_k ^n \wedge f_j ^n)_n, $ and so  $f_k ^n \wedge f_j ^n = 0 $ for all $n \in \N$ and for all $k \neq j$. In addition, since the linear embedding of $E_n$ into $E$ is continuous, 
we have that  $f_k ^n \cvfe 0,$   when $k \to \infty,$ in $E_n'$. This implies that  $|f_k ^n| \cvfe 0$ as $k \to \infty$ for all $n$.

Let $x = (x_n) \in E$. We claim that $|f_k|(x) \to 0.$
First since $f_k \cvfe 0$ in $E'$, there exists $M>0$ such that
$$ \norma{|f_k ^n|} = \norma{f_k ^n} \leq \sum_n \norma{f_k ^n} = \norma{|f_k|}_1 \leq M \quad \text{for all $n,k \in \N$.} $$
Finally, let $\epsilon > 0$. As $\norma{x_n} \to 0$, there exists $n_0 \in \N$ such that
$ \norma{x_n} < \epsilon/(2M) \quad \text{for all } n > n_0. $
On the  other hand, as $|f_k ^n|(x) \to 0$ when $k \to \infty$ for all $n \in \N,$  for each $n \in \N$, there exists $k_n \in \N$ such that
$$ ||f_k ^n|(x)| < \epsilon/(2n_0) \quad \text{for all }k \geq k_n. $$
Let $k_0 = \max\{k_1, \dots, k_{n_0}\}$. If $k \geq k_0$,
\begin{align*}
    ||f_k|(x)|    & \leq \sum_{n=1}^{n_0} ||f_k ^n|(x_n)| + \sum_{n=n_0+1}^\infty ||f_k ^n|(x_n)|     \leq \sum_{n=1}^{n_0} ||f_k ^n|(x_n)| + \sum_{n=n_0+1}^\infty \norma{f_k ^n} \norma{x_n} \\
                & < \sum_{n=1}^{n_0} \epsilon/(2n_0) + \norma{|f_k|}_1 \epsilon/(2M) <\epsilon.
\end{align*}
This means that $|f_k| \cvfe 0$ as wanted.

Now, we assume that $E_n$ does not have property (d) for some $n$.
Then, there exists a disjoint weak* null sequence $(f_n ^k)_k \subset E_n'$ such that 
 $(|f_n ^k|)_k$ is not weak* null. Without loss of generality, assume that there exist $\epsilon > 0$ and an element $x_n \geq 0$ in $E_n$ such that $|f_n ^k|(x_n) \geq \epsilon$ for all $k$. Consider $f_k = (0, \dots, 0, f_n ^k, 0, \dots) \in E'$ for all $k \in \N$. We have that
$$ |f_k| \wedge |f_j| = (0, \dots, 0, |f_n ^k| \wedge |f_j ^k|, 0, \dots) = 0, $$
which implies that $(f_k)$ is disjoint in $E'$. Further, if $x = (x_i) \in E$,
$f_k(x) = f_n ^k(x_n) \to 0$ as $k \to \infty$. Hence $(f_k) \subset E'$ is weak* null. However, 
$$|f_k|(0, \dots, 0, x_n, 0, \dots) = |f_n ^k|(x_n) \geq \epsilon$$
for all $k \in \N$. Thus $E$ lacks property (d). \qed

\medskip
 
Using the same arguments of Proposition \ref{expropd2}  and doing the necessary adaptations we can prove the next result.

\begin{prop} \label{expropd3}
Let $(E_n)$ be a family of Banach lattices and  let $1 \leq p < \infty$. Then  $E = (\bigoplus_{n \in \N} E_n)_p$  has property (d) if, and only if   $E_n$  has property (d) for all $n.$
\end{prop}

Let $E_n = E$ for each $n.$ Then $(\bigoplus_{n \in \N} E_n)_0$ and $(\bigoplus_{n \in \N} E_n)_p$ are denoted by $c_0(E)$ and $\ell_p(E)$, respectively, for all $1 \leq p < \infty$. As a consequence of  Proposition \ref{expropd2} and Proposition \ref{expropd3}, we get the following result.

\begin{corol}
For a Banach lattice $E$, the following are equivalent:
\begin{enumerate}
    \item $E$ has  property (d).
    \item $c_0(E)$ has property (d).
    \item $\ell_p(E)$ has  property (d) for some $1 \leq p < \infty$.
    \item $\ell_p(E)$ has property (d) for all $1 \leq p < \infty$.
\end{enumerate}
\end{corol}

\smallskip

We observe that every bounded linear  operator $T: E \to c_0$ is uniquely determined by a weak* null sequence $(x_n') \subset E'$ such that $T(x) = (x_n'(x))$ for all $x \in E.$ When this  sequence is disjoint, we  say that $T$ is a \textit{disjoint operator}.
\smallskip

The next  Theorem characterizes property (d) with the regularity of a disjoint  linear operator on $E.$ We recall that a linear operator between Banach lattices is said to be \textit{regular} if it can be written as difference of two positive operators, or equivalently, if it is dominated by a positive operator.

\begin{teo} \label{propd8}
A Banach lattice $E$ has property (d) if and only if every disjoint linear  operator on $E$ is regular.
\end{teo}

\proof
Assume that $E$ has property (d) and let $T: E \to c_0$ be a disjoint linear  operator.  Then there is a disjoint sequence $(x_n') \subset E'$ with  $T(x) = (x_n'(x))$ for all $n$. Since $E$ has property (d), we can consider the positive operator $S: E \to c_0$ defined by $S(x) = (|x_n'|(x))$ for all $x \in E$. It follows that $S \geq T$, thus $T$ is regular.

Let $(x_n') \subset E'$ be a disjoint weak* null sequence and let $T: E \to c_0$ be the disjoint operator such that $T(x) = (x_n'(x))$. By hypothesis, $T$ is regular. So we can find a positive linear operator $S \geq T$. Let $(y_n') \subset (E')^+$ be the weak* null sequence such that $S(x) = (y_n'(x))$ for all $x$. In particular, if $x \in E^+$, then $x_n'(x) \leq y_n'(x)$ for all $n$. Thus $x_n' \leq y_n'$  and  $(x_n')^+ \leq y_n'$. Consequently, $(x_n')^+ \cvfe 0$ in $E'$. Hence $(x_n')^- = (x_n')^+ - x_n' \cvfe 0$ in $E'$. Therefore, $|x_n'| \cvfe 0$.
\qed

\medskip

We remark that property (d) is not preserved by closed sublattices. For instance, $E = \ell_\infty$ has property (d), however $F = c$ is a closed sublattice of $E$ without property (d). Yet, we have that property (d) is preserved by projection bands.  Recall that a band $B$ in $E$ is said to be a \textit{projection band} if there exists a linear projection $P: E \to E$ satisfying $P(E) = B$ and $0 \leq Px \leq x$ for all $x \in E^+$. Such projection is called a \textit{band projection} (see Definition 1.2.1 in \cite{meyer}).


\begin{prop} \label{propd10}
Let $E$ be Banach lattice with property (d). If $F$ is a projection band in $E$, then $F$ has property (d).
\end{prop}

\begin{proof}
Let $P: E \to E$ denote the band projection associated to the projection band $F$. In particular, $0 \leq Px \leq x$ holds for all $x \in E^+$. We claim that $P$ is interval preserving, i.e. $[0,Px] = P([0,x])$ for all $x \in E^+$. Indeed, if $y \in [0, Px]$, i.e. $0 \leq y \leq Px$, we get that $y \in F$, because $Px \in F$ and $F$ is an ideal in $E$, and so $Py = y$. Since $Px \leq x$, we have that $0 \leq Py = y \leq Px \leq x$, which yields that $y \in P([0,x])$. Thus $[0, Px] \subset P([0,x])$. On the other hand, if $y \in P([0,x])$, there exists $z \in [0,x]$ such that $y = Pz$, which implies that $0 \leq y = Pz \leq Px$, hence $y \in [0, Px]$. Then $P$ is interval preserving. Consequently, $P'$ is a Riesz homomorphism (see Theorem 1.4.19 in \cite{meyer}). 
In order to prove that $F$ has property (d), let $(y_n')$ be a disjoint weak*null sequence in $F'$. As $P'$ is a Riesz homomorphism, we get that $(P'y_n')$ is a disjoint weak*null sequence in $E'$, and so $|P'y_n'| \cvf 0$ in $E'$. Finally, given $x \in F^+$, we have that
\begin{align*}
        |y_n'|(x) & = \sup \conj{|y_n'(y)|}{|y| \leq x, \, y \in F} = \sup \conj{|y_n' \circ P (y)|}{|y| \leq x, \, y \in F} \\ 
        & \leq \sup \conj{|P'y_n' (y)|}{|y| \leq x, \, y \in E} = |P'y_n'|(x) \to 0.
\end{align*}
Then $|y_n'| \cvfe 0$ in $F'$. \qed

\end{proof} 

\medskip

As a consequence of Proposition \ref{propd10}, we have the following example:

\begin{ex}
Let $E = \left (\bigoplus_{n \in \N} \ell_2^n \right )_0$. It follows from a commentary after Lemma 10 in \cite{loumir} that $E^\perp = \conj{f \in E'''}{f(x) = 0, \, \forall x \in E}$ is a projection band in $E'''$. Thus, by Proposition \ref{propd10}, $E^\perp$ has property (d).
\end{ex}

\smallskip

Let $E$ be a Banach lattice and let $F$ be a closed ideal in $E$. The quotient space $E/F$ is a Banach lattice endowed with the following partial order: $[x] \leq [y]$ if and only if there exist $x_1 \in [x]$ and $y_1 \in [y]$ such that $x_1 \leq y_1$ (see \cite[p. 100]{alip}). Recall that $(E/F)'$ is isometric isomorphic to the annihilator $F^\perp$ (see Proposition 2.6 in \cite{habala}). As a consequence, we get the next Theorem which is a generalization of Proposition \ref{propd10}.

\begin{teo} \label{propd9}
Let $E$ be a Banach lattice and let $F$ be a closed ideal in $E$. If $E$ has property (d), then $E/F$ has property (d).
\end{teo}

\begin{proof}
Let $\pi: E \to E/F$ be the canonical projection and consider its dual operator $\pi': (E/F)' \to F^\perp$, which is an isometric isomorphism. We show that $\pi'$ is a Riesz homomorphism, i.e. $|\pi'f| = \pi'|f|$ for all $f \in (E/F)'$. Indeed, given $f \in (E/F)'$ and $x \in E^+$, it follows from Theorem 1.18 in \cite{alip} that
\begin{align*}
    |\pi'f|(x) & = \sup \conj{|\pi'f(y)|}{|y| \leq x}  = \sup \conj{|f([y])|}{|[y]| \leq [x]} 
             = |f|([x]) = (\pi'|f|)(x).
\end{align*}

If $(f_n) \subset (E/F)'$ be a disjoint weak* null sequence, as $\pi'$ is a Riesz homomorphism, then $(\pi'f_n)$ is a disjoint sequence in $F^\perp$ and so in $E'$. If $x \in E$, we have that
$ \pi'f_n (x)  = f_n([x]) \to 0. $
Thus $(\pi'f_n)$ is a disjoint weak* null sequence in $E'$. As $E$ has property (d), $\pi'|f_n| = |\pi'f_n| \cvfe 0$ in $E'$. If $[x] \in E/F$, we have that
$ |f_n|([x])= \pi'|f_n|(x) \to 0. $
Hence $|f_n| \cvfe 0$ and consequently $E/F$ has property (d). \qed
\end{proof}

\medskip

It follows from Theorem \ref{propd9} that $\ell_\infty/c_0$ has property (d), while it is not $\sigma$-Dedekind complete (see Remark 1.5 in \cite{wnuk}). This shows that property (d) is, in a sense, a more robust concept that $\sigma$-Dedekind completeness.

\medskip

We conclude this section with the following question:

\medskip

\noindent \textbf{Question: } Let $E$ be a Banach lattice with property (d), is it true that  $\ell_\infty(E)$ has property (d)?

\section{The class of  the almost limited completely continuous operators}

The study of certain geometric properties in Banach spaces is largely connected with observing the behavior of bounded linear operators. Whenever a new class or a new property appears in Banach spaces, it is natural to study the relationship of such  property in the context of
operators. The objective here is to study some operator classes from properties in Banach lattices, more specifically, we will study the class of the almost limited completely continuous operators which is connected to the sGP property in Banach lattices.

Let us recall some necessary definitions.
We recall that $T: X \to Y$ is a \textit{completely continuous} linear operator if $Tx_n \to 0$ for every weakly null sequence $(x_n) \subset X$. Since the limited and weakly convergent sequences play an important role in the study of Banach spaces with the GP property, the authors in \cite{salimi} introduced and studied the class of limited completely continuous operators.
Following \cite{salimi}, a linear operator $T: X \to Y$ is \textit{limited completely continuous} (lcc) if $Tx_n \to 0$ for every limited weakly null sequence $(x_n) \subset X$. The class of completely continuous operators and the class of   limited completely continuous operators will be denoted by $Cc(X;Y)$ and $Lcc(X;Y)$, respectively.

As the lattice structure is, in general, distinct of the Banach space, in the setting of Banach lattices is natural study known properties in the class of Banach spaces under a Banach lattice point of view. By considering almost limited sequences instead of limited sequences, the authors in \cite{ardakani} introduced the class of alcc operators and studied its relation with the sGP property. Following \cite{ardakani}, a bounded linear operator $T: E \to Y$ is said to be an \textit{almost limited completely continuous} (alcc operator) if $Tx_n \to 0$ in $Y$ for every almost limited weakly null sequence $(x_n) \subset E$. This class is denoted by $L^a cc (E;Y)$. Moreover, the authors in \cite{ardakani} showed that if $E$ has the sGP property (resp. $E$ has the wDP* property), then $L^a cc(E;Y) = L(E;Y)$ (resp. $L^a cc (E;Y) = Cc(E;Y)$) for every Banach space $Y$ (see Theorem 3.4 and page 22 in \cite{ardakani}, respectively).

In the Banach space theory, it is usual to connect classes of operators with classes of sets. In \cite{salimi2}, the authors introduced the class of the L-limited sets and studied its relations with the class of lcc operators, a subset $A$ of a dual space $X'$ is a \textit{L-limited set} if every limited weakly null sequence $(x_n)$ in $X$ converges uniformly to zero on $A$.

In the same way the authors in \cite{ardakani} introduced the alcc operators in order to establish a "lattice" version of the lcc operators by using almost limited sequences instead of limited sequences, we are going to consider almost limited sequences in the definition above in the place of limited sequences. Following this idea, we introduce the class of strong L-limited sets.

\begin{df} \label{lal1}
A subset $A$ of a dual space $E'$ is called an \textit{strong L-limited set} if every almost limited weakly null sequence $(x_n)$ in $E$ converges uniformly to zero on $A$. 
\end{df}

It is clear that  every strong L-limited subset of $E'$ is L-limited. We will see after Proposition \ref{lal3} that there exist L-limited sets which are not strong L-limited. 

\begin{ex}
Let $E = C(K)$ where $K$ is a compact metric space. It follows from Theorem 2.12 in \cite{ardakani} that every almost limited weakly null sequence in $E$ is norm null. Consequently, $B_{E'}$ is an strong L-limited set.
\end{ex}


\begin{rema} \label{lal2}
Let $E$ be a Banach lattice and let $Y$ be a Banach space.
A bounded linear operator $T: E \to Y$ is alcc if and only if $T'(B_{Y'})$ is a strong L-limited set in $E'$. 

\end{rema}

In \cite{ardakani} the authors showed that if $E'$ has a weak unit or if $E$
has order continuous norm, then $E$ has the sGP if and only if every almost limited weakly null sequence is norm null. From this fact, the following result is immediate:

\begin{prop} \label{lal3}
Assume that $E'$ has a weak unit or $E$ has order continuous norm. Then $E$ has the sGP property if and only if $B_{E'}$ is an strong L-limited set.
\end{prop}

As a consequence of Proposition \ref{lal3}, we have that $B_{L_p[0,1]}$ with $1 < p \leq \infty$ is not strong L-limited, even though they are L-limited (see Theorem 2.3 in \cite{salimi2}). Moreover, since $c_0$ and $\ell_p$ with $1 \leq p < \infty$ have order continuous norm and the sGP property (see page 17 in \cite{ardakani}), Proposition \ref{lal3} yields that $B_{\ell_p}$ with $1 \leq p \leq \infty$ is an strong L-limited set.

\smallskip

We observe that $B_{L_{\infty} [0,1]} = \sol{\{ \chi_{[0,1]}\}}$. Thus, the solid hull of a strong L-limited set is not necessarily strong L-limited. In the next Theorem, we give conditions on $E$ so that the solid hull of an strong L-limited also is strong L-limited.

\begin{teo} \label{lal4}
Let $E$ be a Banach lattice with order continuous norm. If the lattice operations in $E$ are sequentially weakly continuous, then the solid hull of an strong L-limited subset of $E'$ also is strong L-limited.
\end{teo}

\proof
Let $A \subset E'$. Assume that $\sol{A}$ is not strong L-limited. Then there exists  an almost limited weakly null sequence $(x_n) \subset E$ such that $\sup_{x' \in \sol{A}} |x'(x_n)| \geq \epsilon$ for all $n$ and for some $\epsilon > 0$. For each $n$, there exists $x_n' \in \sol{A}$ such that $|x_n'(x_n)| \geq \epsilon$. In particular, for each $n$, we can find $y_n' \in A$ such that $|x_n'| \leq |y_n'|$. 
On the other hand, by Theorem 1.23 of \cite{alip}, there exists $(z_n'') \subset E''$ such that $|z_n''| \leq J_E(|x_n|)$ in $E''$ and $z_n''(y_n') = J_E(|x_n|)(y_n')$ for every $n$. Since $E$ has order continuous norm, $E$ is an ideal of $E''$ (see Theorem 4.9 in \cite{alip}). Then for each $n$, there exists $z_n \in E$ such that $z_n'' = J_E(z_n)$. 
Thus
$$ \epsilon \leq |x_n'(x_n)| \leq |x_n'|(|x_n|) \leq |y_n'|(|x_n|) = z_n^{''}(y_n') = y_n'(z_n) $$
holds for all $n$. 
Since $|z_n| \leq |x_n|$ for all $n$ and $|x_n| \cvf 0$ in $E$, by assumption we have that $z_n \cvf 0$ in $E$. As $E$ has order continuous norm, it has property (d), by Lemma \ref{lemadom1}, $(z_n)$ is almost limited. Hence $A$ cannot be strong L-limited, a contradiction.
\qed

\medskip


    
    


It is known that every weakly compact operator is lcc (see Corollary 2.5 in \cite{salimi}). As a consequence, every relatively weakly compact set in the dual of a Banach space is L-limited (see \cite{salimi2}).  We observe that, if $E$ is a non discrete reflexive Banach lattice (e.g. $L_p[0,1]$ with $1 < p < \infty$), then $E$ does not have the sGP property  by Theorem 2.8 in \cite{ardakani}. On the other hand, as $E$ has order continuous norm, we have that $I_E$ cannot be alcc. As a consequence, $B_{E'}$ is a relatively weakly compact subset of $E'$ which is not strong L-limited. In Proposition \ref{lal6} and Proposition \ref{lal7}, we establish connections between strong L-limited sets, relatively weakly compact sets, alcc operators and weakly compact operators.

\begin{prop} \label{lal6}
Let $E$ be a Banach lattice.
The following assertions are equivalent:
\begin{enumerate}
    \item Every relatively weakly compact set in $E'$ is strong L-limited.
    
    \item For every Banach space $Y$, every weakly compact operator $T: E \to Y$ is alcc.
\end{enumerate}
\end{prop}

\proof
$(1) \Rightarrow (2)$ Let $T: E \to Y$ be a weakly compact operator. So $T'(B_{Y'})$ is relatively weakly compact in $E'$ and, as consequence, a strong L-limited set. Therefore $T$ is an alcc operator.

$(2) \Rightarrow (1)$ Assume that there exists a relatively weakly compact set $A \subset E'$ that is not strong L-limited. Without loss of generality, we can assume that there exists $\epsilon > 0$, $(x_n') \subset A$ and an almost limited weakly null sequence such that $|x_n'(x_n)| \geq \epsilon$ for all $n$. Since $A$ is relatively weakly compact, there exists a subsequence $(x_{ n_k}')$ which converges weakly to $x' \in E'$. Define the weakly compact operator $T: E \to c_0$ by $T(x) = ((x_{n_k}' - x')(x))_n$ (see Theorem 5.26 in \cite{alip}). On the other hand, let $k_0 \in \N$ such that $|x'(x_{n_k})| < \epsilon/2$ for all $k \geq k_0$.
If $k \geq k_0$,
$$ \norma{Tx_{n_k}} \geq |x_{n_k}'(x_{n_k}) - x'(x_{n_k})| \geq \epsilon/2. $$
Thus $T$ is not an alcc operator.
\qed

\medskip

Using the same arguments of Theorem 2.8 in \cite{salimi2} and doing the necessary adaptations, we can prove the following Proposition.

\begin{prop} \label{lal7} Let $E$ be a Banach lattice.
The following assertions are equivalent:
\begin{enumerate}
    \item Every strong L-limited set in $E'$ is relatively weakly compact.
    
    \item For every Banach space $Y$, every alcc operator  $T: E \to Y$ is weakly compact.
\end{enumerate}
\end{prop}

From Propositions \ref{lal6} and \ref{lal7}, we get the next Corollary.

\begin{corol} \label{lal8}
Let $E$ be a Banach lattice. The class of strong L-limited sets coincides with the class of relatively weakly compact sets in $E'$ if and only if  $L^a_{cc}(E;Y)=W(E;Y)$, for every Banach space $Y$.
\end{corol}

Recall that a linear operator $T: E \to F$ is said to be order bounded if it takes order bounded sets from $E$ onto order bounded sets in $F$.
In the next result, we prove that every order bounded linear  operator from a Banach lattice with property (d) into a Banach lattice with the sGP property is alcc.

\begin{prop} \label{alcc0}
Let $E$ and $F$ be Banach lattices with   property (d).  Let $T: E \to F$ be an order bounded linear operator.  If  $F$ has the sGP property, then $T$ is  an alcc operator.
\end{prop}

\proof
Let $(x_n) \subset E$ be an almost limited weakly null sequence and let $A = \sol{(x_n)_n}$.  By Lemma \ref{lemadom1} and Theorem \ref{lemadom2}, we have  that  $T(A)$ is almost limited. As $(Tx_n) \subset T(A)$, $(Tx_n)$ is an almost limited sequence in $F$. Then $Tx_n \cvf 0$ in $F$, and so $Tx_n \to 0$ in $Y$ since $F$ has  the sGP property.
\qed

\smallskip


Now, we give some remarks concerning Proposition \ref{alcc0}.

\begin{remas} \label{exalcc}
\begin{enumerate}
    \item Its is known that for every $1 \leq p \leq \infty$, the Banach lattice $E = L_p[0,1]$ has property (d). Therefore if $F$ is any Banach lattice with both property (d) and the sGP property, we get from Proposition \ref{alcc0} that $L^r(E; F) \subset L^a cc(E; F)$.

    \item Let $E$ be a Banach lattice that contains a lattice copy of $\ell_1$. Then, there exists a positive projection $P: E \to \ell_1$ (see Proposition 2.3.11 in \cite{meyer}). As a consequence of Proposition \ref{alcc0}, we get that $P$ is an alcc operator. In particular, Banach lattices with the positive Schur property contains lattice copy of $\ell_1$ (see page 19 in \cite{wnuksurv}). Besides, we get from Remark \ref{lal2} that $P'(B_{\ell_\infty})$ is an strong L-limited set in $E'$.

    \item There is a Banach lattice $F$ with property (d) which does not have the sGP property, such that, for any Banach lattice $E$, every order bounded operator $T: E \to F$ is alcc. {\normalfont Indeed, let $K$ be a compact metric space and assume that $K$ is $\sigma$-Stonian (i.e. the closure of every open $F_\sigma$ set is open). It follows from Theorem 2.12 in \cite{ardakani} that $F = C(K)$ does not have the sGP property. On the other hand, Proposition 2.1.4 in \cite{meyer} yields that $F = C(K)$ is $\sigma$-Dedekind complete, so it has property (d). Given a Banach lattice $E$, we claim that every bounded operator $T: E \to F$ is completely continuous, hence alcc. Indeed, if $x_n \cvf 0$ in $E$, we have that $Tx_n \cvf 0$ in $F$. In particular, $(Tx_n)$ is a bounded sequence. As $C(K)$ has a unit, $(Tx_n)$ is contained in an order interval. Since $C(K)$ has property (d), it follows from Proposition \ref{prop22} that order intervals are almost limited sets. Consequently, $(Tx_n)$ is an almost limited weakly null sequence. Theorem 2.12 in \cite{ardakani} yields that $\norma{Tx_n} \to 0$.}
\end{enumerate}
\end{remas}

\smallskip

 The next example shows that the assumption of $T$ being order bounded linear operator in Proposition \ref{alcc0} is essential.

\begin{ex}
Let $(r_k)$ denote the Rademacher functions and let
$T: L_1[0,1] \to c_0$ be the operator given by
$$  T (f) = \left (\int_0^1 f(t) \, r_k(t) \, dt \right )_k.
$$
We observe that $T$ is well-defined since $r_n \cvf 0$ in $L_1[0,1]$ and that $T$ is not an alcc operator, once the sequence $(r_n)$ is almost limited  weakly null in $ L_1[0,1]$, but
$ \norma{ T(r_n)}_\infty = \sup_{k \in \N} \left | \int_0^1 r_n(t)\, r_k(t) dt \right | \geq 1 \quad \text{for all } n \in \N. $ We claim that $T$
is not an order bounded operator. Indeed,
as $c_0$ is Dedekind complete, it suffices to check that $T$ is not regular. Suppose that $T$ is regular.  So we assume that there exists $S \geq 0$ with $S \geq T$. The components of $S$ lead to a weak* null positive sequence $(\varphi_n) \subset (L_1[0,1])'$ such that $S(f) = (\varphi_n(f))_n$. Since $(L_1[0,1])' = L_\infty [0,1]$, there exists $(f_n) \subset (L_\infty [0,1])^+$ with $\varphi_n(f) = \int_0^1 f_n(t) \, f(t) \, dt$. 
Since $T\leq S,$ we get that $r_n \leq f_n$ in $L_\infty[0,1]$ which means that $f_n \geq r_n^+$. But
$$ \varphi_n (1) = \int_0^1 f_n(t) \, dt \geq \int_0^1 r_n^+(t) \, dt = 1/2 $$
for all $n$, a contradiction.
\end{ex}

In the class of the  positive linear operators in  Banach lattices, we have a dominated type problem. For instance, let $S,T: E \to F$ be  positive operators such that $S \leq T$. The question is, if $T$ has some property, does $S$ also has it? Here we are interested in the class of alcc operators 
and we see that in general the answer is negative, as the next example shows.

\begin{ex} \label{ex1}
Let $R_1: c_0 \to \ell_\infty$ $S_1, S_2: L_1[0,1] \to \ell_\infty$
be given by
$$ R_1(\alpha_i)_i = (\alpha_i)_i, \quad S_2(f) = \left ( \int_0^1 f dx, \int_0^1 f dx, \, \dots \right ) \quad \text{ and } \quad  S_1(f) = \left (\int_0^1 f r_i^+ (x) dx \right )_i. $$
Define $S,T: c_0 \oplus L_1[0,1] \to \ell_\infty$ by
$$S((\alpha_i), f) = R_1((\alpha_i)) + S_1(f) \quad \text{ and } \quad T((\alpha_i), f) = R_1((\alpha_i)) + S_2(f).$$
Note that $0 \leq S \leq T$. 

We claim that $T$ is an alcc operator. If
 $(x_n,f_n) \subset c_0 \oplus L_1[0,1]$ is an almost limited weakly null sequence, then $(x_n) \subset c_0$ and $(f_n) \subset L_1[0,1]$ are both almost limited weakly null sequences. As $c_0$ has the sGP property, we have that  $\norma{x_n}_\infty \to 0$. So,
$$ \norma{R_1(x_n)} = \norma{x_n}_\infty \to 0. $$
On the  other hand, as  $f_n \cvf 0$ in $L_1[0,1]$, it follows that $\int_0^1 f_n \to 0$, what implies that
$$ \norma{S_2(f_n)}_\infty = \sup_i \left | \int_0^1 f_n \right | = \left | \int_0^1 f_n \right | \to 0.$$
Therefore
$$ T(x_n, f_n) = R_1(x_n) + S_2(f_n) \to 0. $$

It is easy to verify that $S$ is not an alcc operator. In fact, $(r_n) \subset L_1[0,1]$ is an almost limited weakly null sequence. However, $\norma{S_1(r_n)}_\infty \geq 1/2$ for all $n$. Hence $(0, r_n)$ is an almost limited weakly null sequence in $c_0 \oplus L_1[0,1]$ such that $S(0,r_n) = S_1(r_n)$ that is not norm null. 

Also note that $R_1$ is not completely continuous, since
$e_n \cvf 0$ in $c_0$ and $\norma{e_n} = 1$. Hence neither $S$ nor $T$ is completely continuous.
\end{ex}

We observe that the order bounded linear operator $S_1: L_1[0,1] \to \ell_\infty$ given in Example \ref{ex1} is a positive non alcc operator. This shows that the assumption of $F$ having the sGP property in Proposition \ref{alcc0} is essential.


\smallskip

The following result was inspired in Theorem 5.89 of \cite{alip}.

\begin{teo} \label{teodom1}
Let  $E$ and $F$ be Banach lattices and let $S,T: E \to F$ be  positive operators such that $S \leq T$ and $T$ is an alcc operator. If  $E$ has property (d) and  the lattice operations in $E$ are weakly sequentially continuous, then $S$ is   an alcc operator.
\end{teo}

\proof
Let $(x_n)$ be an almost limited weakly null sequence. By Lemma \ref{lemadom1}, $\sol{(x_n)_n}$ is an almost limited set in $E$. In particular, $(|x_n|)$ is an almost limited sequence. On the other hand, since the lattice operations in $E$ are weakly sequentially continuous, $|x_n| \cv 0$.
As $T$ is alcc, $T|x_n| \to 0$. Thus $ Sx_n \leq S|x_n| \leq T|x_n| \to 0.$
\qed

\medskip

The following theorem was inspired by Kalton-Saab theorem's (Theorem 5.90 in  \cite{alip}).

\begin{teo} \label{teodom2}
Let  $E$ and $F$ be Banach lattices and let $S,T: E \to F$ be  positive operators such that $S \leq T$ and $T$ is an  alcc operator. If  $E$ is $\sigma$-Dedekind complete and  $F$ has order continuous norm, then $S$ is  an alcc operator.
\end{teo}

\proof
Let $(x_n) \subset E$ be an almost limited weakly null sequence. We want to prove that $Sx_n \to 0$ in $F$.  
Let $A = \sol{(x_n)_n}$. We observe that for every disjoint sequence $(y_n) \subset A$, $\norma{Ty_n} \to 0.$ [Indeed, by Theorem 3.34 of \cite{alip}, $y_n \cvf 0$ in $E$. However, since $A$ is almost limited (Lemma \ref{lemadom1}), we get that $(y_n)$ is an almost limited weakly null sequence and the statement holds since $T$ is alcc.] Therefore, given $\epsilon > 0$, by Theorem 4.36 of \cite{alip}, we have that there exists $0 \leq u \in E_A$ such that
\begin{equation} \label{eqdom1}
    \norma{T((|x_n| - u^+))} < \epsilon/6 \quad \text{for all $n \in \N$}.
\end{equation}

Now, since $E$ is $\sigma$-Dedekind complete and $F$ has order continuous norm, by Theorem 4.87 of \cite{alip}, there exist positive operators $M_1, \dots, M_k: E \to E$ and $P_1, \dots, P_k: F \to F$ such that  $0 \leq \sum_{i=1}^k P_i \, T \, M_i \leq T$ on $E$ and 
\begin{equation} \label{eqdom2}
    \norma{|S - \sum_{i=1}^k P_i \, T \, M_i| \, (u)} < \epsilon/3.
\end{equation}

We claim  that $TM_i(x_n) \to 0,$ where  $n \to \infty$ for all $i = 1, \dots, k$.
Since $T$ is  an alcc operator, it suffices to prove that $(M_i(x_n))$ is an almost limited weakly null sequence in $E$.
As $E$ has property (d), $A \subset E$ is a solid almost limited set, $M_i: E \to E$ is a  positive operator and  
$M_i(x_n) \cvf 0$ in $E$. It follows from Theorem \ref{lemadom2} that $M_i(A)$ is an almost limited subset of $E$. 
In particular, $(M_i(x_n))_n$ is an almost limited weakly null sequence for each $i=1, \dots, k$ as we stated.  Consequently,
$$ \norma{\sum_{i=1}^k P_i T M_i (x_n)} \leq \sum_{i=1}^k \norma{P_i} \norma{TM_i(x_n)} \to 0. $$
Thus, there exists $n_0 \in \N$ such that
\begin{equation} \label{eqdom3}
    \norma{\sum_{i=1}^k P_i T M_i (x_n)} < \epsilon/3 \quad \text{for all $n \geq n_0$}.
\end{equation}

Finally, by using the same argument used in the proof of Kalton-Saab's theorem (Theorem 5.90 in \cite{alip}), we get that $\norma{Sx_n} < \epsilon$ for all $n \geq n_0$. Here, we give a short version of the proof. Since $0 \leq S \leq T$ and $0 \leq \sum_{i=1}^k P_i T M_i \leq T$, we have that
\begin{align*}
    |S(x_n) - \sum_{i=1}^k P_i T M_i(x_n)| 
        & \leq |S - \sum_{i=1}^k P_i T M_i| (|x_n| - u)^+ + |S - \sum_{i=1}^k P_i T M_i|(u) \\
        & \leq 2 T(|x_n| - u)^+ + |S - \sum_{i=1}^k P_i T M_i|(u).
\end{align*}
Now, from (\ref{eqdom1}) and (\ref{eqdom2}),
$$ \norma{S(x_n) - \sum_{i=1}^k P_i T M_i} \leq 2 \epsilon/3. $$
Finally, $\norma{Sx_n} < \epsilon$ for all $n \geq n_0$ holds from above inequality and (\ref{eqdom3}).
\qed

\medskip

Next, we will study when an alcc operator is completely continuous or compact. Since every compact operator is completely continuous, it follows that
$K(E;Y) \subset Cc(E; Y) \subset L^a cc(E;Y)$. If every alcc operator is a compact operator, we have that
$ K(E;Y) = Cc(E; Y) = L^a cc(E;Y). $
So $Cc(E; Y) = L^a cc(E;Y)$ must be a necessary condition for $K(E;Y) = L^a cc(E;Y)$.In the next proposition we given conditions so that these spaces are distinct.


\begin{prop} \label{alcc2}
Let $E$ be an infinite dimensional reflexive lattice with the sGP property. Then there exists a Banach space $Z$ such that $K(E;Z) \neq L^acc (E;Z)$. 
\end{prop}

\noindent \textit{Proof: }
By above commentary, $E$ cannot have the wDP* property and hence it neither can have the DP* property. By Theorem 2.8 of \cite{salimi}, there exists a Banach space $Z$ such that $Cc(E;Z) \subsetneq Lcc(E;Z)$.
On the  other hand, as $E$ has the sGP property, then
$L^a cc(E;Y) = Lcc (E;Y) = L(E;Y)$
for all Banach space $Y$. It follows that
$K(E;Z) = Cc(E;Z) \neq L^a cc(E;Z)$.
\qed

\smallskip

It is known that if $E$  has the wDP* property, then every alcc operator on $E$ is completely continuous (see page 22 in \cite{ardakani}). 
Using Proposition \ref{alcc0}, we will show that if $E$ has the property (d) and if every alcc operator on $E$ is completely continuous, then $E$ has the wDP* property.

\begin{prop} \label{alcc3}
Let  $E$ be  Banach lattice  with property (d). If  $Cc(E;F) = L^a cc(E;F)$ for all Banach lattices $F, $ then $E$ has the wDP* property.
 \end{prop}

\begin{proof}
Assume that $E$ has property (d) and suppose by contradiction that $E$ does not have the wDP* property. By Theorem 3.1 of \cite{elbour2}, there exist $(x_n) \subset E^+$ a weakly null sequence and $(f_n) \subset (E')^+$ a weak* null sequence such that $f_n(x_n) \geq \epsilon$ for all $n$ and some $\epsilon > 0$. Consider the positive linear operator  $T: E \to c_0$ defined by $T(x) = (f_n(x))_n$ for all $x \in E$, by Proposition \ref{alcc0}  we have  that $T$ is an alcc operator. But $T$ is not completely continuous, once  $\norma{Tx_n} \geq  |f_n(x_n)| \geq \epsilon$ for all $n$. \qed
\end{proof}

\medskip

Corollary 5 of \cite{guti} shows that every completely continuous linear  operator on a Banach space $X$ is compact if and only if $X$ does not contain an isomorphic copy of $\ell_1$. In particular, if $E$ is a Banach lattice that does not contain copy of $\ell_1$ and if $E$ has the wDP* property, then every alcc operator on $E$ must be compact. 

\smallskip

Recall that a linear operator $T: E \to Y$ is called \textit{AM-compact} if it maps order intervals in $E$ to relatively compact sets in $Y$. In the next proposition, we study a relation between the class of alcc operators and the class of AM-compact operators.

\begin{prop}
Let $E$ and $F$ be two Banach lattices. If $E$ has property (d), then every alcc operator $T: E \to F$ is AM-compact.
\end{prop}

\begin{proof}
Let $T: E \to F$ be an alcc operator and let $A \subset E$ be an order interval. It follows from Proposition \ref{prop22} that $A$ is an almost limited set, and so $T(A)$ is a relatively compact subset of $F$ (see Theorem 3.2 in \cite{ardakani}). Thus $T$ is an AM-compact operator.\qed
\end{proof}



    

\medskip

\medskip

\noindent \textbf{Acknowledgments: } The authors are thankful to the referee for the careful reading and considered suggestions leading to a better paper.

\end{document}